\documentclass[a4paper, reqno]{amsart}
\usepackage{amssymb,amsmath,amsfonts,amsthm,mathrsfs}
\usepackage[colorlinks=true, citecolor=blue, anchorcolor=red]{hyperref}
\usepackage{enumerate}

\newcommand{\supp}{\operatorname{supp}}
\newcommand{\R}{\mathbb R}
\newcommand{\N}{\mathbb N}
\newcommand{\C}{\mathbb C}
\newcommand{\Z}{\mathbb Z}

\renewcommand{\Re}{\mathop{\text{\upshape{Re}}}}
\renewcommand{\Im}{\mathop{\text{\upshape{Im}}}}
\newcommand{\cl}{\text{\upshape{cl}}}

\newcommand{\op}{\mathop{\text{\upshape{op}}}}
\renewcommand{\div}{\mathop{\text{\upshape{div}}}}
\newcommand{\rank}{\mathop{\text{\upshape{rank}}}}
\newcommand{\ms}[1]{\mathscr{#1}}

\renewcommand{\Im}{\operatorname{Im}}

\newcommand{\diag}{\operatorname{diag}}

\renewcommand{\epsilon}{\varepsilon}

\newcommand{\norm}[1]{\left\lVert#1\right\rVert}

\renewcommand{\tilde}{\widetilde}
\renewcommand{\phi}{\varphi}

\newtheorem{theorem}{Theorem}[section]
\newtheorem{lemma}[theorem]{Lemma}
\newtheorem{corollary}[theorem]{Corollary}

\theoremstyle{definition}
\newtheorem{definition}[theorem]{Definition}
\newtheorem{remark}[theorem]{Remark}
\newtheorem{example}[theorem]{Example}

\numberwithin{equation}{section}

\begin{document}

\title[Dispersive mixed-order systems in $L^p$-Sobolev spaces]
{Dispersive mixed-order systems in $L^p$-Sobolev spaces and application to the thermoelastic plate equation}
\author{Robert Denk}
\address{Universit\"at Konstanz, Fachbereich f\"ur Mathematik und Statistik,
         78457 Konstanz, Germany}
\email{robert.denk@uni-konstanz.de}
\author{Felix Hummel}
\address{Universit\"at Konstanz, Fachbereich f\"ur Mathematik und Statistik,
         78457 Konstanz, Germany}
         \email{felix.hummel@uni-konstanz.de}
\thanks{}
\date{January 30, 2018}
\begin{abstract}
We study dispersive mixed-order systems of pseudodifferential operators in the setting of $L^p$-Sobolev spaces. Under the weak condition of quasi-hyperbolicity, these operators generate a semigroup in the space of tempered distributions. However, if the basic space is a tuple of $L^p$-Sobolev spaces, a strongly continuous semigroup is in many cases only generated if $p=2$ or $n=1$. The results are applied to the linear thermoelastic plate equation inertial term and with Fourier's or Maxwell-Cattaneo's law of heat conduction.
\end{abstract}



\maketitle

\section{Introduction}

Our investigation is motivated by the analysis of the linear thermoelastic plate equation in the whole space which is given by
\begin{equation}
  \label{1-1}
  \begin{aligned}
    u_{tt} + \Delta^2 u -\mu \Delta u_{tt} + \Delta \theta & = 0 \quad\text{ in }(0,\infty)\times\R^n,\\
    \theta_t + \div q - \Delta u_t & = 0 \quad\text{ in }(0,\infty)\times\R^n,\\
    \tau q_t + q+\nabla\theta & = 0 \quad\text{ in }(0,\infty)\times\R^n,
  \end{aligned}
\end{equation}
supplemented by initial conditions. In \eqref{1-1}, the unknown functions are $u$, $\theta$, and $q$, where $u$ describes the elongation of a plate and $\theta$ and $q$ model the temperature (relative to a fixed reference temperature) and the heat flux, respectively. The parameters $\tau,\mu\ge0 $ are chosen depending on the underlying model. For $\mu>0$, an inertial term is included, for $\tau =0$ the classical Fourier law of heat conduction is assumed, while for $\tau>0$ we take the Cattaneo-Maxwell law. System \eqref{1-1} was investigated in many papers, in particular in the setting of $L^2$-Sobolev spaces. We refer, e.g., to the papers by Lasiecka and Triggiani (\cite{lasiecka-triggiani98}, \cite{lasiecka-triggiani98a}), Racke and Ueda \cite{racke-ueda16}, Said-Houari \cite{said-houari13}, and Ueda, Duan, and Kawashima \cite{ueda-duan-kawashima12} and the references therein.

The aim of the present paper is to study system \eqref{1-1} and general mixed-order systems of pseudodifferential operators in the setting of $L^p$-Sobolev spaces for $p\not=2$. It is well known that the wave equation is well-posed in $L^p$ if and only if $n=1$ (see Littman \cite{littman63}, Peral \cite{peral80}).  Well-posedness in the $L^p$-setting for symmetric hyperbolic systems  was investigated by Brenner \cite{brenner66}. For such systems, the symbol has the form  $ a(\xi) = i\sum_{j=1}^n \xi_j a_j$ with symmetric matrices $a_j\in\R^{N\times N}$, and it was shown that such a system gives  raise to a well-posed Cauchy problem in $L^p$  if and only if the matrices $a_1,\dots, a_n$ commute (\cite{brenner66}, Theorem~1). In the present paper, we study more general mixed-order systems with symbol $a(\xi) = (a_{ij}(\xi))_{i,j=1,\dots,N}$ where each entry belongs to the H\"ormander symbol class $S_{\cl}^{\mu_{ij}}(\R^n)$ of classical pseudodifferential operators of order $\mu_{ij}$. In order to solve the Cauchy problem
\[ \big(\partial_t - a(D) \big) u(t) = 0 \; (t>0),\quad u(0)=u_0,\]
in $\R^n$, one has to study the symbol $e^{ta(\xi)}$. If the equation is quasi-hyperbolic (or correct in the sense of Petrovski\u{\i}), the operator generates a locally uniformly bounded semigroup in the space $\mathscr S'(\R^n;\C^N)$ of tempered distributions (see Theorem~\ref{2.2} below). For a survey on distributional Cauchy problems, we refer to the monograph by Ortner and Wagner \cite{ortner-wagner15}.
The generation of a strongly continuous semigroup (or, equivalently, the well-posedness of the Cauchy problem) in $L^p$-Sobolev spaces can be described by a condition on the multiplier norm of the symbol $e^{ta(\cdot)}$, see Theorem~\ref{2.7}. This result is a slight generalization of  classical results by Brenner \cite{brenner66} and H\"ormander \cite{hoermander60}. We remark that the symbol $e^{ta(\cdot)}$ can  formally  also be seen as the symbol of a Fourier integral operator with matrix-valued and complex phase function. For the scalar (and homogeneous) case of phase functions, many results are known on dispersive estimates in $L^p$, see, e.g., Ruzhansky \cite{ruzhansky01}, and Coriasco and Ruzhansky \cite{coriasco-ruzhansky14}. The main problem in our case and for \eqref{1-1} is the mixed-order structure of the system.

The basic space for a general mixed-order system will be of the form $X_p=\prod_{j=1}^N H_p^{s_j}(\R^n)$, where $H_p^s(\R^n)$ stands for the standard (Bessel potential) Sobolev space. If $s$ is an integer, this space coincides with the classical Sobolev space $W_p^s(\R^n)$. By real interpolation, also Sobolev-Slobedeckii spaces  $ W_p^s(\R^n)$ for non-integer $s$ and Besov spaces $B_{pq}^s(\R^n)$ can be considered. One of the main results of this paper, Theorem~\ref{3.9} below, states that dispersive mixed-order systems generate a $C_0$-semigroup in $X_p$ only in special cases.
In particular, after order reduction due to the definition of the space $X_p$, the operator has to be of order one. Even if this holds, there are restrictions on the eigenvalues if $n>1$. Roughly speaking, the general picture which is known for symmetric hyperbolic systems carries over to more general mixed-order systems.

In Section~4, we apply the above results to the thermoelastic plate equation \eqref{1-1}. In the case $\tau=\mu=0$, it is known that the related operator even generates an analytic semigroup in $L^p$ for every $p\in (1,\infty)$ (see Denk and Racke \cite{denk-racke06}). For the Cattaneo-Maxwell setting $\tau>0$, a $C_0$-semigroup is generated in $L^p$, $p\not=2$ only in the case $n=1$ and $\mu>0$
 (Theorem~\ref{4.3} and Theorem~\ref{4.4}). We remark here that for $\tau>0$ and $\mu=0$ the Cauchy problem is not well-posed even for $n=1$.

For the Fourier law $\tau=0$ (and $\mu>0$), the generation of $C_0$-semigroups again holds if and only if $n=1$ (Theorem~\ref{4.5}). This result cannot be obtained by a straightforward application of the general results, as the relevant part of the symbol is still a combination of first- and second-order. The only nontrivial eigenvalue of the principal symbol (which is of second order) has negative real part which does not lead to a contradiction with the generation of a semigroup. Therefore, to prove Theorem~\ref{4.5}, we explicitly apply an approximate diagonalization procedure (up to operators of order 0) which is motivated by the method in Kozhevnikov \cite{kozhevnikov96} (see also Denk, Saal, and Seiler \cite{denk-saal-seiler09}). This procedure gives a separation of the first-order and the second-order part of the symbol which yields the results on well-posedness in $L^p$.

\section{Well-posedness of the Cauchy problem}

In the following, let $\op[a]$ be a mixed-order  $N\times N$-system of pseudo-differential operators in $\R^n$ with $x$-independent symbols, i.e., $a=(a_{ij})_{i,j=1,\dots,N}$, where $a_{ij}\in S^{\mu_{ij}}(\R^n)$, $\mu_{ij}\in\R$. Here, $S^\mu(\R^n)=S^\mu_{1,0}(\R^n)$ stands for the standard H\"ormander class of $x$-independent symbols of order $\mu\in\R$, i.e., $S^\mu(\R^n)$ is the set of all smooth complex-valued functions $b\in C^\infty(\R^n)$ such that for each $\alpha\in\N_0^n$ there exists a $C_\alpha>0$ satisfying
\[ |\partial_\xi^\alpha b(\xi)|\le C_\alpha \langle \xi\rangle^{\mu-|\alpha|}\quad (\xi\in\R^n).\]
Here we have used the standard multi-index notation $\partial_\xi^\alpha =\partial_{\xi_1}^{\alpha_1}\ldots \partial_{\xi_n}^{\alpha_n}$ and have set $\langle \xi\rangle:=(1+|\xi|^2)^{1/2}$. Note that in this situation we have $a\in S^\mu(\R^n;\C^{N\times N})$ with $\mu:=\max_{i,j=1,\dots,n} \mu_{ij}$. As usual, the pseudo-differential operator related to the symbol $a$ is defined by
$ \op[a]\phi = \mathscr F^{-1} a\mathscr F\phi$ for all $\phi$ belonging to the $\C^N$-valued Schwartz space $\mathscr S(\R^n;\C^N)$. In the above formula, $\mathscr F$ stands for the Fourier transform which is defined by
\[ (\mathscr F\phi)(\xi) := \hat\phi(\xi) := (2\pi)^{-n/2} \int_{\R^n} e^{ix\cdot \xi} \phi(x)dx\quad (\xi\in\R^n)\]
for $\phi\in \mathscr S(\R^n;\C^N)$ and by duality extended to the space of tempered $\C^N$-valued distributions $\mathscr S'(\R^n;\C^N) := L(\mathscr S(\R^n); \C^N)$.

Let $\mathcal O_M(\R^n;\C^{N\times N})$ denote the space of all slowly increasing smooth functions, i.e., the space of all $a\in C^\infty(\R^n;\C^{N\times N})$ for which for each $\alpha\in\N_0^n$ there exist $C_\alpha, m_\alpha>0$ such that
\[ |\partial^\alpha a(\xi)|_{\C^{N\times N}} \le C_\alpha \langle\xi\rangle^{m_\alpha}\quad (\xi\in\R^n).\]
By definition of the H\"ormander class, we  have $S^\mu(\R^n;\C^{N\times N})\subset \mathcal O_M(\R^n;\C^{N\times N})$. It was shown in \cite{amann03}, Thm.~1.6.4, that for $a\in \mathcal O_M(\R^n;\C^{N\times N})$ the multiplication operator $\phi\mapsto  a\phi$ is a continuous linear operator belonging to   $L(\mathscr S(\R^n;\C^N))$. Moreover, there exists a unique hypocontinuous and bilinear map \[ \mathcal O_M(\R^n;\C^{N\times N})\times \mathscr S'(\R^n;\C^N) \to \mathscr S'(\R^n;\C^N),\quad (a,u)\mapsto au,\]
induced by the dual pairing
\[ \langle au,\phi\rangle_{\mathscr S'(\R^n;\C^N)\times \mathscr S(\R^n;\C^N)} = \langle u, a^\top \phi\rangle_{\mathscr S'(\R^n;\C^N)\times \mathscr S(\R^n;\C^N)} = \sum_{j=1}^N u_j\Big( \sum_{k=1}^N a_{kj}\phi_k\Big)\]
(\cite{amann03}, Thm.~1.6.4). Therefore, for $a\in \mathcal O_M(\R^n;\C^{N\times N})$, we obtain by this duality an operator $\op[a]\in L(\mathscr S'(\R^n;\C^N))$ (cf. also \cite{amann03}, Remark~1.9.11).
For $a\in S^\mu(\R^n;\C^{N\times N})$, we  consider the Cauchy problem
\begin{equation}
  \label{2-1}
  \begin{aligned}
  \partial_t u - \op[a] u & = 0\quad (t>0),\\
  u(0) & = u_0.
  \end{aligned}
\end{equation}
The following definition of quasi-hyperbolicity is classical and can be found, e.g., in  \cite{ortner-wagner90}. This condition is also called correct in the sense of Petrovski\u{\i} or Petrovski\u{\i} condition, see \cite{MR1155843}, Definition~2 on p.~168, and \cite{hoermander83},  p.~143.

\begin{definition}
  \label{2.1}
  Let $a\in \mathcal O_M(\R^n;\C^{N\times N})$. Then the Cauchy problem \eqref{2-1} is called quasi-hyperbolic if there exists a constant $M_a\in\R$ such that
  \begin{equation}
    \label{2-2}
    \det\big(\lambda-a(\xi))\not=0\quad (\Re\lambda>M_a,\, \xi\in\R^n).
  \end{equation}
\end{definition}
For a survey on distributional Cauchy problems and fundamental solutions, we mention  the monograph \cite{ortner-wagner15}. For differential operators, the following result can be found in \cite{bargetz-ortner15}, Proposition~3.

\begin{theorem}
  \label{2.2} Let $a\in S^\mu(\R^n;\C^{N\times N})$, and assume that equation \eqref{2-1} is quasi-hyperbolic. Then for every $u_0\in\mathscr S'(\R^n;\C^N)$ there exists a unique solution $u\in C^1([0,\infty);\mathscr S'(\R^n;\C^N))$ of \eqref{2-1}. This solution is given by $u(t) = \op[e^{t a(\cdot)}]u_0$. Moreover, the family $(T(t))_{t\ge 0}$ with $T(t) := \op[e^{t a(\cdot)}]$ is a locally uniformly bounded semigroup on $\mathscr S'(\R^n;\C^N)$. The analog results hold with $\mathscr S'(\R^n;\C^N)$ being replaced by $\mathscr S(\R^n;\C^N)$.
\end{theorem}

\begin{proof}
By \eqref{2-2}, we see that for all $\lambda\in\C$ with $\Re\lambda > M_a+1$, all eigenvalues of the matrix $\lambda-a(\xi)$ have real part not less than $1$. Therefore, $|\det(\lambda-a(\xi))|\ge 1$ if $\Re\lambda\ge M_a+1$. By Cramer's rule, every entry of the matrix $(\lambda-a(\xi))^{-1}$ is a quotient of the form $\frac{c_{ij}(\xi,\lambda)}{\det(\lambda-a(\xi))}$ where $c_{ij}(\xi,\lambda)$ stands for the cofactor.

It is well-known from the theory of pseudo-differential operators that sums and products of scalar symbols in $S^*:=\bigcup_{\mu\in\R} S^\mu(\R^n)$ belong to $S^*$ again. We also remark that derivatives with respect to $\xi$ of $\frac{c_{ij}(\xi,\lambda)}{\det(\lambda-a(\xi))}$ are again of the form $\frac{\tilde c_{ij}(\xi,\lambda)}{(\det(\lambda-a(\xi)))^m}$ with some $m\in\N$, where $\tilde c_{ij}$ depend  polynomially on the entries of the matrix $\lambda-a(\xi)$.

From this and the above estimate on the determinant we obtain
\[ (\lambda-a(\xi))^{-1} \in S^{\tilde \mu}(\R^n;\C^{N\times N})\quad (\Re\lambda> M_a+1)\]
for some $\tilde\mu\in\R$. In particular, $(\lambda-a(\xi))^{-1}\in \mathcal O_M(\R^n;\C^{N\times N})$ if $\Re\lambda >M_a+1$. Therefore, Theorem~3.1.3 in \cite{amann03} can be applied which states that there exists a unique fundamental solution for the Cauchy problem \eqref{2-1}.
This implies that \eqref{2-1} has a unique distributional solution (cf. \cite{amann03}, Theorem~3.1.1). On this other hand, by \cite{amann03}, Remark~3.2.3(c), $T(t):= \op[e^{ta(\cdot)}]$ defines a locally uniformly bounded semigroup $(T(t))_{t\ge 0}$ on $\mathscr S'(\R^n;\C^{N})$ and on $\mathscr S(\R^n;\C^N)$, and the unique solution $u$ of \eqref{2-1} is given by $u(t)=T(t)u_0$ for $t>0$.
\end{proof}

Whereas well-posedness of the Cauchy problem \eqref{2-1} holds in $\mathscr S'(\R^n;\C^N)$ under very weak assumptions, the situation is different if we consider $\op[a]$ as an unbounded operator in some Banach space $X\subset \mathscr S'(\R^n;\C^N)$. In particular, we are interested in the case $X=L^p(\R^n;\C^N)$.

\begin{definition}
  \label{2.3}
 Let $X$ be a Banach space with norm $\|\cdot\|$, and let $A\colon X\supset D(A)\rightarrow X$ be a closed and densely defined linear operator. Then the Cauchy problem
\begin{equation}\label{2-3}
\begin{aligned}
\partial_tu-Au& =0\quad(t>0)\\
	    	   u(0)&=u_0
\end{aligned}
\end{equation}
is called well-posed if for every $u_0\in D(A)$  there exists a unique (classical) solution $u\in C^1([0,\infty),X)$ of \eqref{2-3} with $u(t)\in D(A)\;(t> 0)$, and if for all $T>0$ there exists a constant $C_T>0$ such that for all $u_0\in D(A)$  we have that
\begin{equation}
  \label{2-4}
\norm{u(t)}\leq C_T\norm{u_0}\quad (t\in [0,T]).
\end{equation}
\end{definition}

It is well known that well-posedness is equivalent to the generation of a $C_0$-semigroup in $X$. On the other hand, this is also equivalent to the existence of a mild solution for all initial values $u_0\in X$. First, we give the definition (see, e.g., \cite{arendt-batty-hieber-neubrander11}, Def.~3.1.1).

\begin{definition}
  \label{2.4} A function $u\in C([0,\infty),X)$ is called a mild solution of the Cauchy problem \eqref{2-3} if for all $t\in [0,\infty)$ we have
  \[ \int_0^t u(s)ds \in D(A) \quad \text{and} \quad A\int_0^t u(s)ds = u(t)-u_0.\]
\end{definition}

\begin{theorem}
  \label{2.5} Let $A\colon X\supset D(A)\to X$ be a closed densely defined linear operator. Then the following statements are equivalent:
  \begin{enumerate}
    [(i)]
    \item The operator $A$ generates a $C_0$-semigroup on $X$.
    \item For all $u_0\in X$ there exists a unique mild solution of \eqref{2-3}.
    \item There exists a subspace $D\subset D(A)$ which is dense in $X$ such that for all $u_0\in D$ the Cauchy problem \eqref{2-3} has a unique classical solution $u$, and for every $T>0$ there exists $C_T>0$ such that \eqref{2-4} holds for all $u_0\in D$.
  \end{enumerate}
\end{theorem}

\begin{proof}
  The equivalence of (i) and (ii) is shown in \cite{arendt-batty-hieber-neubrander11}, Theorem~3.1.12. In the same theorem, it is also shown that (i) implies well-posedness of the Cauchy problem \eqref{2-3}, i.e. (iii) holds with $D:= D(A)$. Therefore, we only have to show that (iii) implies (ii).

 Let $u_0\in X$. Since $D$ is dense in $X$, there is a sequence $(u_{k,0})_{k\in\N}\subset D$ such that $u_{k,0}\rightarrow u_0$ in $X$ as $k\rightarrow\infty$. For all $k\in\N$ let $u_k\in C^1([0,\infty),X)$ be the classical solution to \eqref{2-3} with initial value $u_{k,0}$. Since $A$ is closed, the $u_k$ are mild solutions (cf. \cite{arendt-batty-hieber-neubrander11}, Prop.~3.1.2). By assumption, for any $T>0$ there is a constant $C_T>0$ such that
 \[
  \norm{u_k(t)-u_\ell(t)}\leq C_T\norm{u_{k,0}-u_{\ell,0}}\rightarrow 0\quad(\ell,k\rightarrow \infty)
 \]
Therefore, $(u_k)_{k\in\N}$ converges uniformly on compact subsets of $[0,\infty)$. We define the limit $u(t):=\lim_{n\rightarrow\infty}u_k(t)\;(t>0)$. Due to the estimate \eqref{2-4} for $u_0\in D$, the definition of $u(t)$ does not depend on the chosen sequence $(u_{k,0})_{k\in\N}\subset D$. Moreover, we get that
\[
 \int_0^tu(s)\,ds=\lim\limits_{k\rightarrow\infty}\int_0^tu_k(s)\,ds
\]
and
\[
 \lim\limits_{k\rightarrow\infty}A\int_0^tu_k(s)\,ds=\lim\limits_{k\rightarrow\infty}(u_k(t)-u_{k,0})=u(t)-u_0
\]
for all $t\geq0$. By the closedness of $A$, we get that $\int_0^tu(s)\,ds\in D(A)$ as well as $A\int_0^tu(s)\,ds=u(t)-u_0$. It remains to show that this mild solution is unique. Let $v\in C([0,\infty),X)$ be another mild solution. Then, for all $t\geq0$ we get
\[
 u(t)-v(t)-A\int_0^t(u(s)-v(s))\,ds=0.
\]
Defining $w(t):=\int_0^t(u(s)-v(s))\,ds$ we obtain a classical solution for the Cauchy problem to the initial value $w(0)=0$. By \eqref{2-4} with $u_0=0\in D$ we obtain $w(t)=0$ for all $t\geq0$ and therefore $u=v$.
\end{proof}

For mixed-order systems of pseudo-differential operators, the closedness of $\op[a]$ holds if we consider the maximal domain. More precisely, for $p\in (1,\infty)$ and $s\in \R$ let $H_p^s(\R^n)$ denote the Bessel potential space $H_p^s(\R^n) := \{u\in \mathscr S'(\R^n): \op[\langle\,\cdot\,\rangle^s]u\in L^p(\R^n)\}$ with its canonical norm. Then we obtain the following result.

\begin{lemma}
  \label{2.6}
  Let $s_1,\dots,s_n\in\R$, and let $X:=\prod_{j=1}^n H_p^{s_j}(\R^n)$. For the symbol $a\in S^\mu(\R^n;\C^{N\times N})$,  define the unbounded operator $A:= \op[a]$ in $X$ with (maximal) domain $D(A) := \{ u\in X: Au\in X\}$. Then $A$ is densely defined and closed.
\end{lemma}

\begin{proof}
  Because of $\mathscr S(\R^n;\C^N)\subset D(A)$, $A$ is densely defined. Let $(u_k)_{k\in\N}\subset D(A)$ with $u_k\to u$ in $X$ and $Au_k\to v$ in $X$ for $k\to\infty$. By the continuity of the embedding $X\subset\mathscr S'(\R^n;\C^N)$, we get that $Au_k\to v$ in $\mathscr S'(\R^n;\C^N)$. On the other hand, we also have $u_k\to u$ in $\mathscr S'(\R^n;\C^N)$, and $\op[a]$ is continuous in $\mathscr S'(\R^n;\C^N)$ which gives $Au_k\to Au$ in $\mathscr S'(\R^n;\C^N)$.

  Since $\ms{S}^{\prime}(\R^n,\C^N)$ is a Hausdorff space, it follows that $Au=v$ in $\ms{S}^{\prime}(\R^n,C^N)$ and by the injectivity of the embedding $X\hookrightarrow \mathscr S'(\R^n;\C^N)$, we get $u\in D(A)$ and $Au=v$ in $X$.
\end{proof}

For the investigation of $C_0$-semigroups in the space $X$ from the previous lemma, the notion of an $L^p$-Fourier multiplier is useful (cf. \cite{hoermander60}, Def.~1.3, and  \cite{brenner66}, Section~2). In the following, we always assume $p\in (1,\infty)$.

\begin{definition}
A function $m\in L^\infty(\R^n;\C^{N\times N})$ is called an $L^p$-Fourier multiplier if there exists a constant $c_p>0$ such that for all $u\in \mathscr S(\R^n;\C^N)$ we have $\op[m]u := \mathscr F^{-1} m \mathscr Fu\in L^p(\R^n;\C^N)$ and
\[ \|\op[m]u\|_{L^p(\R^n;\C^N)} \le c_p \|u\|_{L^p(\R^n;\C^N)}\quad (u\in \mathscr S(\R^n;\C^N)).\]
In this case, $\op[m]$ extends by continuity to a bounded linear operator $\op[m]\in L(L^p(\R^n;\C^N))$. We denote by $M_p^N$ the space of all $L^p$-Fourier multipliers and endow $M_p^N$ with the norm $\|m\|_{M_p^N} := \|\op[m]\|_{L(L^p(\R^n;\C^N))}$.
\end{definition}

A similar version of the following theorem was proved in (\cite{brenner73}, Lemma~5.1). However, the notion of well-posedness therein only requires initial data in $C_0^{\infty}(\R^n)$ to have a unique solution, whereas in the semigroup-theoretic notion of well-posedness all initial data in the domain of the generator should have a unique solution. It should also be noted that this theorem was proved in (\cite{arendt-batty-hieber-neubrander11}, Proposition~8.1.3) using Laplace transform techniques.
\begin{theorem}
  \label{2.7}
  Let $a\in S^\mu(\R^n;\C^{N\times N})$ be quasi-hyperbolic, and define the unbounded operator $A$ in $X:=L^p(\R^n;\C^N)$ by $A=\op[a]$ and $D(A) := \{u\in X: Au\in X\}$. Then the Cauchy problem \eqref{2-3} is well-posed if and only if for all $T>0$ there is a $C_T>0$ such that
   \begin{equation}
     \label{2-5}
     \big\|e^{ta(\cdot)}\big\|_{M_p^N}\leq C_T\quad (t\in [0,T]).
   \end{equation}
   In this case, the semigroup generated by $A$ is given by $(\op[e^{ta(\cdot)}])_{t\geq0}$.
\end{theorem}

\begin{proof}
  (i) Let \eqref{2-3} be well-posed, and let $u_0\in D(A)$. From Theorem~\ref{2.2}, we know that $u(t) := \op[e^{ta(\cdot)}]u_0$ is the unique solution in $\mathscr S'(\R^n;\C^N)$. By the definition of well-posedness, for $T>0$ there exists a $C_T>0$ such that
  \[ \|u(t)\|_{L^p(\R^n;\C^N)} = \|\op[e^{ta(\cdot)}]u_0\|_{L^p(\R^n;\C^N)} \le C_T \|u_0\|_{L^p(\R^n;\C^N)} \quad (t\in [0,T]).\]
  As $\mathscr S(\R^n;\C^N)\subset D(A)$, the function $e^{ta(\cdot)}$ is an $L^p$-Fourier multiplier, and its multiplier norm satisfies  \eqref{2-5}.

  (ii) Assume now that \eqref{2-5} holds. We define $m(t,\xi):= e^{ta(\xi)}$ and
 fix the initial value $u_0\in \mathscr S(\R^n;\C^N)$. By Theorem~\ref{2.2}, for $u(t):= \op[e^{ta(\cdot)}]u_0$ we obtain $u\in C^1([0,\infty),\mathscr S(\R^n,\C^N))$ with
\begin{equation}
 \label{2-7}
 \partial_t u(t) = \op[a] u(t) = \op[a] \op[e^{ta(\cdot)}] u_0 =  \op[e^{ta(\cdot)}]\op[a] u_0=  \op[e^{ta(\cdot)}] A u_0
\end{equation}
in $\mathscr S(\R^n,\C^N)$ for every $t\ge 0$. An iteration gives $\partial_t^2 u(t) =  \op[e^{ta(\cdot)}] A^2 u_0$.
 In particular, we obtain $u(t)\in L^p(\R^n;\C^N)$ and $Au(t)\in L^p(\R^n;\C^N)$ and therefore $u(t)\in D(A)$ for every $t\ge 0$.

Applying twice the fundamental theorem of calculus, we get as equality in $\mathscr S(\R^n;\C^N)$ for $t,h\ge 0$:
\begin{equation}\label{2-6}
  \tfrac 1h\big( u(t+h) - u(t)\big) - (\op[a]u) (t)  =   \int_0^1\int_0^1 s h \op[e^{(t+rsh)a(\cdot)}] A^2 u_0 dr\,ds.
\end{equation}
By assumption, $\|e^{ta(\cdot)}\|_{M_p^N}$ is uniformly bounded on bounded intervals. Therefore, we can estimate the $L^p$-norm of the right-hand side of \eqref{2-6} by
\[ \int_0^1 \int_0^1 s |h| \Big(\sup_{\tau\in [t,t+h]} \|e^{\tau a(\cdot)}\|_{M_p^N}\Big) \|A^2 u_0\|_{L^p(\R^n;\C^N)} dr\,ds \le C| h| \|A^2 u_0\|_{L^p(\R^n;\C^N)}.\]
The same argument holds for $t\ge 0$ and $h<0$ with $t+h\ge 0$. Therefore, we see that the left-hand side of \eqref{2-6} tends to zero in $L^p(\R^n;\C^N)$ for $h\to 0$. Consequently, we have $\partial_t u(t) = A u(t)$ in $L^p(\R^n;\C^N)$ for every $t\ge 0$.

In particular,  the above differentiability yields $u\in C([0,\infty),L^p(\R^n;\C^N))$. Due to the identity \eqref{2-7}, $\partial_t u$ is a solution of the Cauchy problem \eqref{2-1} with $u_0$ being replaced by $Au_0$. Therefore, $\partial_t u$ is continuous, too, and we have that $u\in C^1([0,\infty),L^p(\R^n;\C^N))$ is a classical solution.

By the assumption \eqref{2-5},
\[
 \norm{u(t)}_{L^p(\R^n;\C^N)}= \| \op[e^{ta(\cdot)}] u_0\|_{L^p(\R^n;\C^N)}\le
 C_T\norm{u_0}_{L^p(\R^n;\C^N)}\quad(t\in [0,T]).
\]
Therefore all assumptions of Theorem~\ref{2.5} (iii) are satisfied with $D = \mathscr S(\R^n;\C^N)$, and by Theorem~\ref{2.5} (i) we see that $A$ generates a $C_0$-semigroup which implies well-posedness of \eqref{2-3}.
\end{proof}

\begin{corollary}
  \label{2.8} If in the situation of Theorem~\ref{2.7} the Cauchy problem \eqref{2-3} is well-posed in $L^p(\R^n;\C^N)$ then it is well-posed in every $L^r(\R^n;\C^N)$ with $r\in [\min\{p,q\}, \max\{p,q\}]$. Here $q$ is the conjugate exponent to $p$, i.e. $\frac 1p + \frac 1q=1$.
\end{corollary}

\begin{proof}
  This follows from the equivalence in Theorem~\ref{2.7} and the fact that $M_p^N = M_q^N \subset M_r^N$, see \cite{hoermander60}, Theorem~1.3.
\end{proof}

\begin{corollary}\label{2.9}
  Let $a\in S^\mu(\R^n;\C^{N\times N})$ be quasi-hyperbolic, and define $A$ in $X:= \prod_{j=1}^N H_p^{s_j}(\R^n)$ as in Lemma~\ref{2.6}. Then the following statements are equivalent:
  \begin{enumerate}
    [(i)]
    \item The Cauchy problem \eqref{2-3} is well-posed in $X= \prod_{j=1}^N H_p^{s_j}(\R^n)$.
    \item Let $\Lambda(\xi) := \diag(\langle\xi\rangle^{s_1},\dots,\langle\xi\rangle^{s_n})$. Then for all $T\ge 0$ there exists a $C_T>0$ such that
        \begin{equation}\label{2-11}
         \big\| \Lambda e^{ta(\cdot)}\Lambda^{-1}\big\|_{M_p^N} \le C_T\quad (t\in [0,T]).
        \end{equation}
    \item For every $s\in \R$, the Cauchy problem is well-posed in $X= \prod_{j=1}^N H_p^{s_j+s}(\R^n)$.
  \end{enumerate}
\end{corollary}

\begin{proof}
  By definition, $\op[\Lambda]\colon X\to L^p(\R^n;\C^N)$ is an isometric isomorphism. Therefore, a function $u\in C^1([0,\infty),X)$ is a classical solution to \eqref{2-3} with the initial value $u_0\in D(A)$ if and only if $\tilde u:= \op[\Lambda]u\in C^1([0,\infty),L^p(\R^n;\C^N))$ is a classical solution of
  \begin{equation}
    \label{2-8}
    \begin{aligned}
      \op[\Lambda^{-1}] \partial_t \tilde u - \op[a] \op[\Lambda^{-1}] \tilde u & = 0 \quad (t>0),\\
      \tilde u(0) & = \tilde u_0
    \end{aligned}
  \end{equation}
  with $\tilde u_0:= \op[\Lambda]u_0$. Also the continuous dependence on the initial value (inequality \eqref{2-4}) is maintained. Applying $\op[\Lambda]$ to the first line in \eqref{2-8}, we see that \eqref{2-3} is well-posed in $X$ if and only if
  \begin{equation}
    \label{2-9}
    \begin{aligned}
      \partial_t \tilde u - \op[\tilde a]\tilde u & = 0\quad (t>0),\\
      \tilde u(0) & = \tilde u_0
    \end{aligned}
  \end{equation}
  is well-posed in $L^p(\R^n;\C^N)$, where $\tilde a:= \Lambda a\Lambda^{-1}$. Now Theorem~\ref{2.7} yields the equivalence of (i) and (ii) if we take into account that
  $\Lambda(\xi) \exp(ta(\xi))\Lambda^{-1}(\xi) = \exp\big[t \Lambda(\xi)a(\xi)\Lambda^{-1}(\xi)\big]$.

  As we have
   \begin{align*}
    \diag\big(\langle\xi\rangle^{s_1+s}, & \ldots,\langle\xi\rangle^{s_N+s}\big)a(\xi)\diag\big(\langle\xi\rangle^{-s_1-s},\ldots,\langle\xi\rangle^{-s_N-s}\big)\\
   &=\diag\big(\langle\xi\rangle^{s_1},\ldots,\langle\xi\rangle^{s_N}\big)a(\xi)\diag\big(\langle\xi\rangle^{-s_1},\ldots,\langle\xi\rangle^{-s_N}\big),
 \end{align*}
 condition in (ii) holds for $s_1,\dots,s_n$ if and only if it holds for $s_1+s,\dots,s_n+s$ for any $s\in\R$. This gives the equivalence of condition (iii) to (i) and (ii).
\end{proof}

\begin{remark}\label{2.11}
a)   For $s\in\R$, $p\in (1,\infty)$, and $q\in[1,\infty]$, let $B_{pq}^s(\R^n)$ denote the standard Besov space. Then, if the conditions of Corollary~\ref{2.9} are satisfied, the Cauchy problem is well-posed in the space $\prod_{j=1}^N B_{pq}^{s_j}(\R^n)$. This follows by real interpolation, as, e.g., $B_{pq}^s(\R^n) = (H_p^{s-1}(\R^n), H_p^{s+1}(\R^n))_{1/2,q}$. In particular, we get well-posedness in the Sobolev-Slobodecki\u{\i} spaces $W_p^s(\R^n) = B_{pp}^s(\R^n)$, $s\not\in\Z$.

b) In the situation of Corollary~\ref{2.9}, consider the perturbed Cauchy problem
\begin{equation}
  \label{2-10}
  \begin{aligned}
    \partial_t u - \op[a] u - \op[b] u & = 0\quad (t>0),\\
    u(0) & = u_0.
  \end{aligned}
\end{equation}
If $a$ satisfies \eqref{2-11} and if $b\colon\R^n\to \C^{N\times N}$ is a function satisfying  $\tilde b:=\Lambda b\Lambda^{-1}\in M_p^N$, then \eqref{2-10} is well-posed in $X=\prod_{j=1}^N H_p^{s_j}(\R^n)$. This follows from the fact that $\op[\tilde a]+\op[\tilde b]$ is a bounded perturbation of $\op[\tilde a]$ in \eqref{2-9}, and the set of generators of $C_0$-semigroups is stable under bounded perturbations (see \cite{arendt-batty-hieber-neubrander11}, Corollary~3.5.6).
\end{remark}

\section{Multipliers and mixed-order systems in $L^p$-spaces}

In this section, we want to investigate in which cases condition \eqref{2-5} from Theorem~\ref{2.7} can hold. We will consider systems of classical (polyhomogeneous) pseudo-differential operators with constant ($x$-independent) coefficients. Therefore, we start with the definition of homogeneity.

\begin{definition}
  \label{3.1} Let $d\in\R$. A function $a\in C(\R^n\setminus\{0\}, \C^{N\times N})$ is called homogeneous of degree $d$ if there exists an $R>0$ such that
  \begin{equation}\label{3-1}
    f(t\xi)=t^d f(\xi)
  \end{equation}
   holds for all $\xi\in\R^n$ with $|\xi|\ge R$ and all $t>1$. If this equality holds for all $\xi\not=0$ and all $t>0$, then $a$ is called strictly homogeneous.

   Let $R>0$, and let $V$ be an open subset of the unit sphere $S^{n-1}:=\{\eta\in\R^n: |\eta|=1\}$.    If \eqref{3-1} holds for all $t>1$ and all $\xi$ in a truncated cone of the form
   \[ S_{R,V} := \{ r\eta: r>R,\, \eta \in V\},\]
    then $a$ is called homogeneous in $S_{R,V}$.
\end{definition}

\begin{remark}
  \label{3.2}
  If $a\in C^{[n/2]+1}(\R^n\setminus\{0\},\C^{N\times N})$ is strictly homogeneous of degree $0$, then every derivative of order $k$ is strictly homogeneous of degree $-k$. Therefore,  $a\in M_p^N$ by the theorem of Mikhlin (see, e.g., \cite{arendt-batty-hieber-neubrander11}, Theorem~E.3).
\end{remark}

We want to compare the multiplier properties of a matrix-valued function and the eigenvalues of the matrix. For this, we start with two remarks on the smoothness of eigenvalues and eigenvectors which should be well known but which we could not find in literature.

\begin{lemma}
  \label{3.3}
  Let $U\subset\R^n$ be open and non-empty, and let $a\in C^\infty(U,\C^{N\times N})$. Then there exists an open non-empty set $\tilde U\subset U$ such that (with appropriate numbering)  the eigenvalues $\lambda_1(\xi),\dots, \lambda_N(\xi)$ of $a(\xi)$ satisfy $\lambda_1,\dots,\lambda_N\in C^\infty(\tilde U)$.
\end{lemma}

\begin{proof}
  We prove the statement by induction on $N\in\N$, the case $N=1$ being trivial. For $\xi\in U$, let $p(\lambda,\xi) := \det(\lambda-a(\xi)) = \prod_{j=1}^N(\lambda-\lambda_j(\xi))$. We may assume that the numbering of the eigenvalues is chosen such that all $\lambda_j$ are continuous (see \cite{rahman-schmeisser02}, Chapter 1.3). By induction, there exists an open non-empty set $U_0\subset U$ and $\tau_1,\dots,\tau_{N-1}\in C^\infty(U_0)$ such that $\partial_\lambda p(\lambda,\xi) = N\prod_{j=1}^{N-1} (\lambda-\tau_j(\xi))$. (Note that for the induction we formally  have to write the zeros of the polynomial  $\partial_\lambda p(\lambda,\cdot)$ as the eigenvalues of its companion  matrix.) By continuity of $\lambda_1,\tau_1,\dots,\tau_{N-1}$, the set
 \[
  W:=\bigcap_{j=1}^{N-1}\{\xi\in U_0:\,\lambda_1(\xi)\neq \tau_j(\xi)\}
 \]
is open. Moreover, the implicit function theorem yields  $\lambda_1\vert_W\in C^{\infty}(W)$. Therefore, if $W\neq\emptyset$ we define $U_1:=W$. If $W=\emptyset$, then we choose $A\subset U_0$ as the closure of a non-void open ball contained in $U_0$. Again by continuity it follows that the sets
\[
 A_j:=\{\xi\in A:\, \lambda_1(\xi)=\tau_j(\xi)\}
\]
for $j=1,\ldots,N-1$ are closed. Moreover, since $W=\emptyset$ we have that that
\[
 A=U_0\cap A=\bigcup_{j=1}^{N-1}A_j.
\]
Therefore, there exits a $j\in\{1,\ldots,N-1\}$ such that $A_j$ contains a non-void open ball $U_1$ (this can be seen, e.g., by an application of the Baire category theorem). As $\lambda_1\vert_{U_1}=\tau_j\vert_{U_1}$, we have  $\lambda_1\in C^{\infty}(U_1)$.

Replacing now $U$ by $U_1$ and repeating the same procedure for $\lambda_2,\ldots,\lambda_N$, we obtain open sets $U_0\supset U_1\supset\ldots\supset U_N\neq\emptyset$ with $\lambda_\ell\in C^{\infty}(U_\ell)$. In particular, we have $\lambda_1,\ldots \lambda_N\in C^{\infty}(U_N)$. Setting $\tilde U:= U_N$, we obtain the statement.
\end{proof}

\begin{lemma}
  \label{3.4}
  Let $U\subset\R^n$ be open and non-empty, and let $a\in C^\infty(U;\C^{N\times N})$ and $\lambda\in C^\infty(U)$ such that $\lambda(\xi)$ is an eigenvalue of $a(\xi)$ for each $\xi\in U$. Then there exists an open non-empty set $\tilde U\subset U$ and a function $v\in C^\infty(\tilde U;\C^N)$ such that $v(\xi)$ is an eigenvector to the eigenvalue $\lambda(\xi)$.
\end{lemma}

\begin{proof}
  We define $b\in C^\infty(U;\C^{N\times N})$ by $b(\xi):= a(\xi)-\lambda(\xi)I_N$, so we have to consider the kernel of $b$. We set  $k:=\max_{\xi\in U} \rank b(\xi)<N $ and choose $\xi_0\in U$ with $\rank b(\xi_0)=k$. Without loss of generality, we may assume that the left upper $k\times k$ corner of $b(\xi_0)$ is invertible. Accordingly, we write
  \[ b(\xi) = \begin{pmatrix}
    b^{(1,1)}(\xi) & b^{(1,2)}(\xi)\\ b^{(2,1)}(\xi) & b^{(2,2)}(\xi)
  \end{pmatrix}\]
  with $b^{(1,1)}\in C^\infty(U;\C^{k\times k})$, $b^{(1,2)}\in C^\infty(U;C^{k\times (N-k)})$, $b^{(2,1)}\in C^\infty(U;\C^{(N-k)\times k})$ and $b^{(2,2)}\in C^\infty(U;\C^{(N-k)\times (N-k)})$.
   By continuity, there exists an open ball $\tilde U\subset U$ such that $b^{(1,1)}(\xi)$ is invertible for all $\xi\in\tilde U$. By the definition of $k$, for all $\xi\in \tilde U$ the last $N-k$ columns are linear combinations of the first $k$ columns. Therefore, we obtain
   \[ b(\xi) = \begin{pmatrix}
    b^{(1,1)}(\xi) & b^{(1,1)}(\xi) c(\xi)\\ b^{(2,1)}(\xi) & b^{(2,1)}(\xi)c(\xi)
  \end{pmatrix}\quad (\xi\in\tilde U),\]
  where $c(\xi) := (b^{(1,1)}(\xi))^{-1} b^{(1,2)}(\xi)$. Note that $c\in C^\infty(\tilde U;\C^{k\times (N-k)})$. Let $e_1$ be the first unit vector in $\C^{N-k}$, and set $v(\xi):= \binom{c(\xi)e_1}{-e_1}\in\C^N$. Then $v(\xi)$ is an eigenvector of $a(\xi)$ to the eigenvalue $\lambda(\xi)$ for all $\xi\in\tilde U$ and depends smoothly on $\xi$.
\end{proof}

The following definition is essentially taken from \cite{brenner66}, p.~30.

\begin{definition}
  \label{3.5}
Let $U\subset\R^n$ be open, and let $m\in L^\infty(U;\C^{N\times N})$. Then $m$ is called a local $L^p$-Fourier multiplier in $U$ if there exists  $\tilde m\in M_p^N$ such that $\tilde m|_U = m$. The space of all such functions will be denoted by $M_p^N(U)$. For $m\in M_p^N(U)$, we define $\|u\|_{M_p^N(U)}$ as the supremum over all $\|\op[\tilde m]f\|_{L^p(\R^n;\C^N)}$ where $f\in\mathscr S(\R^n;\C^N)$ with $\supp \mathscr Ff\subset U$ and $\|f\|_{L^p(\R^n;\C^N)}\le 1$ and where $\tilde m\in M_p^N$ with $\tilde m|_U = m$.
\end{definition}

\begin{remark}\label{3.6}
  a) Let $a\in C^\infty(\R^n\setminus\{0\},\C^{N\times N})$. If $a$ is homogeneous of degree $d\le 0$, then $a$ is a local $L^p$-Fourier multiplier in $U:=\{\xi\in\R^n: |\xi|>\epsilon\}$ for all $\epsilon>0$ by the theorem of Mikhlin, applied to a smooth extension of $a|_U$. Similarly, if $a$ is strictly homogeneous of degree $d\ge 0$, then $a$ is a local $L^p$-Fourier multiplier in $\{\xi\in\R^n\setminus\{0\}: |\xi| < R\}$ for all $R >0$.

  b) Let $R>0$ and let $V\subset S^{n-1}$ be open. If $a\in C^\infty(S_{R,V},\C^{N\times N})$ is homogeneous in $S_{R,V}$ of degree $d\in\R$, then there exists an open subset $\tilde V\subset V$ such that the eigenvalues and eigenvectors of $a$  are smooth and homogeneous in $S_{R,\tilde V}$ of degree $d$. In fact, by Lemma~3.3 and 3.4, there exists an open set $U_0\subset S_{R,V}$ where the eigenvalues and eigenvectors are smooth. We choose $r_0>R$ such that $V_0:= r_0S^{n-1}\cap U_0\not=\emptyset$ and set $\tilde V:= r_0^{-1} V_0 \subset S^{n-1}$. Then we can extend the eigenvalues and eigenvectors by homogeneity from $V_0$ to $S_{R,\tilde V}$.
\end{remark}

\begin{theorem}
  \label{3.7}
  In the situation of Theorem~\ref{2.7}, assume that $a\in S^\mu(\R^n,\C^{N\times N})$ is homogeneous of degree $\mu\in\R$.

  a) If $\mu\le 0$, then the Cauchy problem \eqref{2-3} is well-posed in $L^p(\R^n,\C^N)$ for all $p\in (1,\infty)$.

  b) Let $\mu >0$, and assume that for sufficiently large $\xi\in\R^n$ all eigenvalues of $a(\xi)$ have negative real part. Then the Cauchy problem \eqref{2-3} is well-posed for all $p\in (1,\infty)$.

  c) Let $\mu>0$, and assume that there exists an eigenvalue $\lambda(\xi)$ of $a(\xi)$ with $\lambda(\xi)\in i\R\setminus\{0\}$ for all $\xi\in S_{R,V}$ for some $R>0$ and some open set $\emptyset \not=V\subset S^{n-1}$. If \eqref{2-3} is well-posed for some $p\not=2$, then $\mu=1$.
\end{theorem}

\begin{proof}
  a) It is well known that $A=\op[a]$ is a bounded operator in $L^p(\R^n,\C^N)$ for all $p\in (1,\infty)$ if $a\in S^0(\R^n,\C^{N\times N})$, see, e.g., \cite{wong99}, Theorem~10.7. Therefore, $A$ generates a $C_0$-semigroup in $L^p(\R^n,\C^N)$.

  b) Under the assumption b), the symbol $a$ satisfies the classical condition of parameter-ellipticity, i.e. we have
  \[ \det (\lambda-a(\xi))\not=0\quad (\Re\lambda\ge 0,\, |\xi|\ge R)\]
  for sufficiently large $R>0$. Therefore, the operator $A$ even generates a holomorphic semigroup and \eqref{2-3} is well-posed in $L^p(\R^n,\C^N)$ (see, e.g., \cite{grubb95}, Theorem~1.7 and Theorem~2.3).

  c) We may assume that $a$ is homogeneous in $S_{R,V}$. Moreover, by Remark~\ref{3.6} b), we may also assume that $\lambda\in C^\infty(S_{R,V})$, that we have an eigenvector $v\in C^\infty(S_{R,V},\C^N)$ and that both $\lambda$ and $v$ are homogeneous in $S_{R,V}$ of degree $\mu$.

  Let $\emptyset\not= U\subset S_{R,V}$ be an open ball. As $v(\xi)\not=0$ for all $\xi\in U$, we can apply \cite{brenner66}, Lemma~4, which yields that there exists an open $\emptyset \not= U_0\subset U$ such that for all $f\in \mathscr S(\R^n)$ with $\supp \hat f\subset U_0$ we have
  \begin{equation}
    \label{3-2}
    \|f\|_{L^p(\R^n)} \le C \|\op[v] f\|_{L^p(\R^n,\C^N)}.
  \end{equation}
  Here, $\op[v]f := (\op[v_1]f,\dots,\op[v_n]f)^\top$ for $v=(v_1,\dots,v_n)^\top$.
  Let $k\in\N$. We choose $f\in \mathscr S(\R^n)$ with $\supp \hat f\subset U_0$. We apply \eqref{3-2} to $\op[e^{k\lambda(\cdot)}]f$ instead of $f$ and obtain
  \begin{align*}
    \big\| \op\big[ & e^{k\lambda(\cdot)}\big] f\big\|_{L^p(\R^n)} \le C \big\|\op\big[ e^{k\lambda(\cdot)}v(\cdot)\big] f\big\|_{L^p(\R^n,\C^N)}\\
    & = C \big\|\op\big[ e^{ka(\cdot)} v(\cdot)\big] f\big\|_{L^p(\R^n,\C^N)}\\
    & = C \big\|\op\big[ e^{a(k^{1/\mu}\,\cdot\,)} v(\cdot)\big] f\big\|_{L^p(\R^n,\C^N)}.
  \end{align*}
  For the last equality, we used the homogeneity of $a$ in $S_{R,V}$ and $\supp\hat f\subset S_{R,V}$ (here we also used $\mu>0$). Let $v_0\in \mathscr D(\R^n)$ be an extension of $v|_{U_0}$. From the elementary fact that
  \[ \big\| e^{a(k^{1/\mu}\,\cdot\,)}\big\|_{M_p^N} = \big\| e^{a(\cdot)}\big\|_{M_p^N}\]
  (see \cite{arendt-batty-hieber-neubrander11}, Proposition~E.2 e)), we obtain
  \begin{align*}
    \big\|   \op\big[  & e^{k\lambda(\cdot)}\big] f\big\|_{L^p(\R^n)} \\
     & \le C \big\| \op\big[e^{a(k^{1/\mu}\,\cdot\,)}\big]\big\|_{L(L^p(\R^n;\C^N))} \| \op[v_0]\|_{L(L^p(\R^n), L^p(\R^n;\C^N))} \| f\|_{L^p(\R^n)}\\
    & \le C \big\| e^{a(\cdot)}\big\|_{M_p^N} \big\|\op[v_0]\big\|_{L(L^p(\R^n), L^p(\R^n;\C^N))} \| f\|_{L^p(\R^n)}\\
    & \le C' \| f\|_{L^p(\R^n)}
  \end{align*}
  with a constant $C'$ depending on $a$ and $v_0$ but not on $f$ or $k$. Note that we have $\lambda(\xi)\in i\R$ and therefore $|e^{\lambda(\xi)}|=1$ for $\xi\in U_0$. Thus, we may apply \cite{brenner66}, Lemma~5, to get
  \begin{equation}
    \label{3-3}
    \lambda(\xi) = i \xi_0^\top \xi + i\lambda_0\quad (\xi\in U_0)
  \end{equation}
  for some $\lambda_0\in\R$ and $\xi_0\in\R^n$. However, as $\lambda$ is homogeneous of degree $\mu>0$, we obtain $\lambda_0=0$ and $\xi_0\not=0$ as well as  $\mu=1$.
\end{proof}

\begin{remark}
  \label{3.8}
  a) The statement in c) also holds if $a\in C^\infty(\R^n\setminus\{0\},\C^{N\times N})$ is strictly homogeneous of degree $\mu>0$, as this is a bounded perturbation of some homogeneous symbol $\tilde a\in C^\infty(\R^n,\C^{N\times N})$ (cf. Remark~\ref{3.6} a)).

  b) Well-posedness is invariant under similarity transformations: Assume that $S\in L(L^p(\R^n;\C^N))$ is an isomorphism. Then the Cauchy problem is well-posed in $L^p(\R^n;\C^N)$ for $\op[a]$ if and only if it is well-posed for $S^{-1}\op[a]S$. This can be seen as in the proof of Corollary~\ref{2.9}. In particular, this holds if $S = \op[s]$ with $s,s^{-1}\in S^0(\R^n;\C^{N\times N})$.

  c) In the proof of Theorem~\ref{3.7} c), we have seen that well-posedness in $L^p$ implies a strong condition on the eigenvalues of $a(\xi)$: Assume in the situation of Theorem~\ref{3.7} that $\mu=1$. If there exists an eigenvalue $\lambda(\xi)\in i\R\setminus\{0\}$ which is of the form $\lambda(\xi) = \lambda^{(0)} |\xi|$ on an open nonempty set, then well-posedness in $L^p$, $p\not=2$, is only possible for $n=1$. In fact, we have seen above that the eigenvalues of $a(\xi)$ have the form \eqref{3-3}. Therefore,  $|\xi|\lambda^{(0)} = \xi_0^\top \xi+ i+\lambda_0 $ for all $\xi$ in a nonempty open set which is only possible for $n=1$. This situation occurs, in particular, if $a(\xi) = |\xi| a^{(0)}$ with a constant matrix $a^{(0)}\in\C^{N\times N}$ with at least one purely  imaginary eigenvalue.
\end{remark}

\begin{remark}
  \label{3.8a}
  In the situation of Theorem~\ref{3.7}, let us consider the particular case that $a$ is homogeneous of degree 1 and linear in $\xi$, i.e., $a(\xi)=\sum_{j=1}^n \xi_j a_j$ with $a_j\in\C^{N\times N}$. To our knowledge, there is no characterization of all matrices which lead to a well-posed problem, and we only state some remarks on this.

  a) For $p=2$ there is in fact a characterization of all matrices which lead to a well-posed problem, see (\cite{arendt-batty-hieber-neubrander11}, Proposition~8.4.2) and references therein.

  b) If $a_j=i\tilde a_j$ with symmetric real matrices $\tilde a_j\in\R^{N\times N}$, then \eqref{2-3} is well-posed in $L^p(\R^n;\C^N)$ if and only if all matrices $\tilde a_1,\dots, \tilde a_n$ commute. This is a classical result by Brenner (\cite{brenner66}, Theorem~1).

  c) If in the above situation all eigenvalues are purely imaginary, then there is a complete characterization of all matrices leading to a well-posed problem which was derived by Brenner in (\cite{brenner73}, Section~5).

  d) If all eigenvalues of $a(\xi)$ have negative real part for large $\xi\in\R^n$, then \eqref{2-3} is well-posed for all $p\in (1,\infty)$ by Theorem~\ref{3.7} b).

  e) If there exists a $j\in\{1,\dots,n\}$ and an eigenvalue $\lambda_0\in i\R$ of $a_j$ with nontrivial Jordan structure, then \eqref{2-3} is not well-posed even for $p=2$. In fact, setting $\xi = (0,\dots,\xi_j,\dots,0)^\top$, we see that the symbol $e^{ta(\xi)} = e^{t\xi_ja_j}$ is unbounded.

  f) Let $n=1$, and let $a(\xi)=\xi a_1$ with a diagonalizable matrix $a_1\in\C^{N\times N}$. Then \eqref{2-3} is well-posed in $L^p$, $p\not=2$, if and only if all eigenvalues of $a_1$ have nonpositive real part.
\end{remark}

The result of Theorem~3.7 extends to a system of classical pseudo-differential operators. Here a symbol $a\in S^\mu(\R^n,\C^{N\times N})$ belongs to the space $S^\mu_{\cl}(\R^n,\C^{N\times N})$ of classical (polyhomogeneous) symbols if there exists an asymptotic expansion $a\sim \sum_{j=0}^\infty a_j$ where $a_j\in S^{\mu-j}(\R^n,\C^{N\times N})$ is homogeneous of degree $\mu-j$. In this case, $a_0$ is called the principal symbol of $a$.

\begin{lemma}
  \label{3.9a} Let $a\in S^\mu_{\cl}(\R^n,\C^{N\times N})$ be quasi-hyperbolic. Then the statement of Theorem~\ref{3.7} c) hold analogously, where now the eigenvalues of the principal symbol $a_0$ have to be considered.
\end{lemma}

\begin{proof}
 Assume that $\mu>0$ and that the Cauchy problem \eqref{2-3} for $\op[a]$ is well-posed in $L^p(\R^n,\C^N)$ for some $p\not=2$. We choose $m\in \N_0$ with $\mu-m>0$ and $\mu-m-1\le 0$. Then $a-\sum_{j=0}^m a_j\in S^0(\R^n,\C^{N\times N})$, and by bounded perturbation (see Remark~\ref{2.11} b)), we may assume that $a = \sum_{j=0}^m a_j$.

  We choose $R>0$ such that $a_0,\dots,a_m$ are homogeneous for $|\xi|\ge R$ and fix $\chi\in C^\infty(\R^n)$ with $\chi=0$ for $|\xi|\le R$ and $\chi=1$ for $|\xi|\ge R+1$. As $\chi\in M_p^N$ by Mikhlin's theorem, we have by Theorem~\ref{2.7}
  \[ \|e^{ta(\cdot)} \chi(\cdot)\|_{M_p^N} \le C_T\quad (t\in [0,T]).\]
  In particular, the same estimate holds if we replace $t$ by $t k^{-\mu} \le t$ for $k\in\N$. By \cite{arendt-batty-hieber-neubrander11}, Proposition~E.2 e) again, we see that
   \[ \|e^{t k^{-\mu} a(k\,\cdot\,)} \chi(k\,\cdot\, )\|_{M_p^N} \le C_T\quad (t\in [0,T]).\]
 For every $\xi\in\R^n\setminus\{0\}$, the homogeneity of $a_j$ yields
 \[ k^{-\mu} a(k\xi) = \sum_{j=0}^m k^{-\mu} a_j(k\xi) = \sum_{j=0}^m k^{-j} a_j^{(h)}(\xi) \to a_0^{(h)}(\xi)\quad (k\to\infty),\]
 where $a_j^{(h)}$ denotes the strictly homogeneous version of $a_j$, i.e. the strictly homogeneous function which coincides for large $\xi$ with $a_j$. We also have $\chi(k\xi)\to 1$ for every $\xi\not=0$. Therefore, the sequence $(\exp(tk^{-\mu} a(k\,\cdot\,)) \chi(k\,\cdot\,))_{k\in\N}$ is a bounded sequence in $M_p^N$ converging pointwise almost everywhere to $\exp(ta_0^{(h)})$. Consequently, $\exp(ta_0^{(h)})\in M_p^N$ and
 \begin{equation}\label{3-4}
  \| e^{ta_0^{(h)}(\cdot)}\|_{M_p^N} \le C_T\quad (t\in [0,T])
  \end{equation}
 (see \cite{arendt-batty-hieber-neubrander11}, Proposition~E.2 f)).

 Since $m(t,\xi):= e^{ta_0(\xi)}-e^{ta_0^{(h)}(\xi)}$ is smooth in $\R^n\setminus\{0\}$ and has compact support, and since $a_0^{(h)}$ is homogeneous of positive degree, it is easy to see that for every $\alpha\in\N_0^n$, the expression $\xi^\alpha\partial_\xi^\alpha m(t,\xi)$ is bounded by a constant independent of $\xi$ and of $t\in [0,T]$. By Mikhlin's theorem, $m(t,\cdot)\in M_p^N$, and from \eqref{3-4} we get
 \[ \| e^{ta_0(\cdot)}\|_{M_p^N} \le C_T\quad (t\in [0,T]).\]
 Therefore, we can apply Theorem~\ref{3.7} c) to $a_0$ and obtain $\mu=1$ if $a_0$ satisfies the assumptions of Theorem~\ref{3.7}.
\end{proof}

We summarize Corollary~\ref{2.9}, Theorem~\ref{3.7}, and the above remarks in the following theorem which is one of the main results of the present paper.

\begin{theorem}
  \label{3.9}
  Let $a=(a_{ij})_{i,j=1,\dots,N}\colon \R^n\to \C^{N\times N}$ be a quasi-hyperbolic mixed-order system of classical pseudodifferential operators with constant coefficients, $a_{ij}\in S^{\mu_{ij}}_{\cl}(\R^n)$. For $p\in (1,\infty)$, let $A_p$ be the realization of $\op[a]$ in the basic space $X_p = \prod_{j=1}^n H_p^{s_j}(\R^n)$ with maximal domain.

  Define $\Lambda(\xi) := \diag(\langle\xi\rangle^{s_1},\dots,\langle\xi\rangle^{s_n})$ and $\tilde a := \Lambda a\Lambda^{-1}\in S_{\cl}^{\mu}(\R^n;\C^{N\times N})$ where $\mu$ is the maximal order of the entries of $\tilde a$. Let $\tilde a_0$ be the principal symbol of $\tilde a$.

  a) If $\mu\le 0$, then the Cauchy problem \eqref{2-3} is well-posed for all $p\in (1,\infty)$.

  b) If $\mu >0$ and for sufficiently large $\xi\in\R^n$ all eigenvalues of $\tilde a_0(\xi)$ have negative real part, then \eqref{2-3} is well-posed for all $p\in (1,\infty)$.

  c) Let $\mu>0$ and assume that there exists an eigenvalue $\lambda(\xi)\in i\R\setminus\{0\}$ of $\tilde a_0(\xi)$ for all $\xi \in S_{R,V} := \{r\eta: r>R,\, v\in V\}$, where $R>0$ and $V$ is an open nonempty set. If \eqref{2-3} is well-posed for some $p\not=2$, then $\mu=1$. Moreover, if $\lambda(\xi)$ only depends on $|\xi|$ for all $\xi\in S_{R,V}$, then well-posedness is only possible if $p=2$ or $n=1$.
\end{theorem}

\begin{proof}
  a) and b) follow in exactly the same way as in the proof of Theorem~\ref{3.7}. The necessity of $\mu=1$ in c) is stated in Lemma~\ref{3.9a}, and the case of eigenvalues depending only on $|\xi|$ is discussed in Remark~\ref{3.8} c), applied to $\tilde a_0$.
\end{proof}

\begin{example}
 As a direct example of the above results, we consider a damped plate equation with $\rho(-\Delta)^\alpha u_t$ for $\alpha\in[0,1]$ and $\rho>0$ as a damping term, i.e. we consider the equation
 \begin{align}\label{DampedPlate}
 \begin{aligned}
  u_{tt}(t,x)+\Delta^2u(t,x)+\rho(-\Delta)^\alpha u_t(t,x)=0\quad&((t,x)\in[0,\infty)\times\R^n),\\
  u(0,x)=u_0(x)\quad &(x\in\R^n),\\
  u_t(0,x)=u_1(x)\quad &(x\in\R^n).
  \end{aligned}
 \end{align}
If we substitute $v:=u_t$ and set $U:= (u, v)^\top$, we obtain the Cauchy problem
\begin{align*}
 (\partial_t-A(D))U(t)=0\quad(t>0),\quad U(0)=U_0
\end{align*}
with
\begin{align*}
 A(D):=\begin{pmatrix}
  0 & 1 \\
  -\Delta^2 & -\rho(-\Delta)^{\alpha}
 \end{pmatrix}
\end{align*}
and $U_0:=(u_0,u_1)^\top$. We take $X_p:=W^{2}_p(\R^n)\times L^p(\R^n)$ as the basic space and $D(A_p):= W^{4}_p(\R^n)\times W^2_p(\R^n)$ as the domain of the realization of $A(D)$ in $X_p$ which is given by
\[
 A_p\colon X_p\supset D(A_p)\to X_p,\;A_pU:=A(D)U.
\]
For $p=2$, this leads to a well-posed problem. This can be seen as a direct application of the Corollary~\ref{2.9} since straightforward calculation shows that condition \ref{2.9} (ii) is satisfied. In particular, this choice of spaces is the natural one. On the other hand, in the case $\alpha=1$ (so-called structural damping), the operator $A_p$ generates an analytic $C_0$-semigroup in $X_p$ and even has maximal $L^p$-regularity for every $p\in (1,\infty)$. This has been proved in \cite[Theorem 2.5]{denk-schnaubelt15}, but we will see below that the well-posedness also immediately follows by Theorem \ref{3.9}.

But first, we turn our attention to the case $p\in(1,\infty)\setminus\{2\}$ and $\alpha\in[0,1)$. Following the approach of Theorem \ref{3.9}, we consider the symbol $\tilde{a}:=\Lambda a \Lambda^{-1}$ where $a$ is the symbol belonging to $A(D)$ and $\Lambda(\xi)=\operatorname{diag}(\langle\xi\rangle^2,1)$. We obtain
\begin{align*}
  \tilde{a}(\xi)=\begin{pmatrix}
                  0 & \langle \xi \rangle^{2} \\
                  -\tfrac{|\xi|^4}{\langle \xi \rangle^2} & -\rho |\xi|^{2\alpha}
                 \end{pmatrix}
\end{align*}
with homogeneous  principal symbol
\begin{align*}
   \tilde{a}_0(\xi)=\begin{pmatrix}
                  0 & |\xi|^2 \\
                  -|\xi|^2 & 0
                 \end{pmatrix}.
\end{align*}
Hence, we have the eigenvalues $\lambda_1(\xi)=i|\xi|^2$ and $\lambda_2(\xi)=-i|\xi|^2$ and it follows from Theorem \ref{3.9} c) that the equation \eqref{DampedPlate} is not well-posed in $X_p$ for $p\neq2$ even in the one-dimensional case.\\
If we take $\alpha=1$, then the principal symbol is given by
\begin{align*}
   \tilde{a}_0(\xi)=\begin{pmatrix}
                  0 & |\xi|^2 \\
                  -|\xi|^2 & -\rho|\xi|^2
                 \end{pmatrix}.
\end{align*}
so that we obtain
$$\lambda_1(\xi)=-\tfrac{1}{2}(\rho-\sqrt{\rho^2-4})|\xi|^2 \quad\text{and} \quad\lambda_2(\xi)=-\tfrac{1}{2}(\rho+\sqrt{\rho^2-4})|\xi|^2$$
as the eigenvalues. Theorem \ref{3.9} b) now implies that the equation \eqref{DampedPlate} is well-posed in $X_p$ for $p\in(1,\infty)$ and $\alpha=1$.
\end{example}

\section{Application to the thermoelastic plate equation}

In this section, we apply the previous results to the thermoelastic plate equation with Fourier and Maxwell-Cattaneo type heat conduction model, respectively. The dissipative structure of this equation in $L^2$-spaces has been studied, e.g., in \cite{racke-ueda16}. Omitting physical constants, the linear thermoelastic plate equation is given by
\begin{equation}
  \label{4-1}
  \begin{aligned}
    u_{tt} + \Delta^2 u - \mu \Delta u_{tt} + \Delta\theta & = 0,\\
    \theta_t + \div q - \Delta u_t & = 0,\\
    \tau q_t + q + \nabla \theta & = 0.
  \end{aligned}
\end{equation}
In \eqref{4-1}, the unknowns are the elongation $u=u(t,x)$ of the plate at time $t\ge 0$ and position $x\in\R^n$, the temperature (difference) $\theta=\theta(t,x)$, and the heat flux $q=q(t,x)$. The two parameters $\mu,\tau\ge 0$ describe whether an inertial term is present ($\mu>0$) and which type of heat conduction model is used ($\tau=0$ for Fourier's law and $\tau>0$ for Cattaneo-Maxwell's law). In the $L^2$-setting, many results are known, for instance on (non-)exponential stability and regularity loss phenomena. For this, we refer to \cite{lasiecka-triggiani98}, \cite{liu-liu97}, \cite{quintanilla-racke08}, \cite{racke-ueda16} and the references therein.

\subsection{Cattaneo-Maxwell's law}

We first consider the case $\tau>0$, i.e., Cattaneo-Maxwell's law of heat conduction. We start with the additional assumption $\mu>0$. In this case, we apply the operator $(1-\mu\Delta)^{-1}$ to the first equation in \eqref{4-1} and set $U:= (u, v, \theta, q)^\top$ with $v:= u_t$. We obtain the Cauchy problem
\begin{equation}\label{4-3}
 (\partial_t - A(D)) U(t) = 0  \; (t >0),\quad U(0)=U_0
\end{equation}
with
\[ A(D) := \begin{pmatrix}
  0 & 1 & 0 & 0 \\
  -(1-\mu\Delta)^{-1} \Delta^2 & 0  & -(1-\mu\Delta)^{-1} \Delta & 0 \\
  0 & \Delta & 0 & -\div \\
  0 & 0 & -\frac 1\tau \nabla & -\frac 1\tau
\end{pmatrix}\]
and $U_0 := (u_0,u_1,\theta_0,q_0)^\top$. The symbol of $A(D)$ is given by
\[ a(\xi) :=
\begin{pmatrix}
			0 & 1 & 0 & 0 &\ldots&0\\
			\frac{-\vert\xi\vert^4}{1+\mu\vert\xi\vert^2} & 0 & \frac{\vert\xi\vert^2}{1+\mu\vert\xi\vert^2} & 0 &\ldots&0\\
			0 & -\vert\xi\vert^2 & 0 & i\xi_1 &\ldots & i\xi_n\\
			0 & 0 & \frac{i\xi_1}{\tau} & -\frac{1}{\tau} &\ldots& 0\\
			\vdots&\vdots&\vdots&\vdots&\ddots&\vdots\\
			0&0&\frac{i\xi_n}{\tau}&0&\ldots&-\frac{1}{\tau}
 \end{pmatrix}.\]
As the basic space for the Cauchy problem, we natural choice is
\[ X_p := W_p^2(\R^n) \times W_p^1(\R^n) \times L^p(\R^n) \times L^p(\R^n;\C^n).\]
The realization of $A(D)$ in $X_p$ is given by the operator $A_p\colon X_p\supset D(A_p)\to X_p$ with maximal domain
\[ D(A_p) := \big\{ U\in X_p: A(D) U\in X_p\big\},\; A_pU := A(D) U.\]
By the structure of the matrix $A(D)$, we immediately obtain
\begin{align*}
 D(A_p) & = \big\{ (u,v,\theta,q)^\top \in W_p^3(\R^n)\times W_p^2(\R^n)\times W_p^1(\R^n)\times L^p(\R^n;\C^n): \\
 & \qquad \div q \in L^p(\R^n)\big\}.
 \end{align*}
We start with some remarks on the $L^2$-case. Part b) of the following lemma shows that the choice of the space $X_p$ essentially is the only possible one even for $p=2$.

\begin{lemma}
  \label{4.1}
  Let $\tau>0$, $\mu>0$, and $p=2$.

  a) The operator $A_2$ generates a $C_0$-semigroup in $X_2$.

  b) For $\mathbf s=(s_1,\dots,s_4)$, let $A_2^{(\mathbf s)}$ be the realization of $A(D)$ in the basic space
  \[ X_2^{(\mathbf s)} := H^{s_1}(\R^n)\times H^{s_2}(\R^n)\times H^{s_3}(\R^n)\times H^{s_4}(\R^n;\C^n)\]
  with maximal domain. If $A_2^{(\mathbf s)}$ generates a $C_0$-semigroup in $X_2^{(\mathbf s)}$, then $\mathbf s = (c+2,c+1,c,c)$ for some $c\in\R$.
\end{lemma}

\begin{proof}
a) This can be seen by a standard application of the Lumer-Phillips theorem where it is convenient to endow $X_2$ with the equivalent norm
\[ \|u\|_{X_2}^2 = \|u\|_{H^2(\R^n)}^2 + \|(1-\mu\Delta)^{-1} v\|_{L^2(\R^n)}^2 + \|\theta\|_{L^2(\R^n)}^2 + \|q\|_{L^2(\R^n;\C^n)}^2.\]
With this norm, it is straightforward to see that the operator
\[ \begin{pmatrix}
			0 & 1 & 0 & 0\\
			-(1-\mu\Delta)^{-1}(\Delta^2-1) & 0 & -(1-\mu\Delta)^{-1}\Delta & 0\\
			0 & \Delta & -1 & -\operatorname{div}\\
			0 & 0 & -\frac{\nabla}{\tau} & -\frac{1}{\tau}
		\end{pmatrix}\]
is a bounded perturbation of $A_2$ and dissipative in $X_2$.

b) For $\mathbf s\in\R^4$ and $\xi\in\R^n$, we define
\[ \Lambda^{(\mathbf s)} (\xi) := \diag \big( \langle \xi\rangle^{s_1},\langle\xi\rangle^{s_2},
\langle\xi\rangle^{s_3}, \langle\xi\rangle^{s_4}, \ldots, \langle\xi\rangle^{s_4}\big)\in\R^{(n+3)\times(n+3)}.\]
Assume that $A^{(\mathbf s)}$ generates a $C_0$-semigroup in $X_2^{(\mathbf s)}$. By the Hille-Yosida theorem, there exists a $C>0$ such that for large $\lambda>0$ we have $\|\lambda(\lambda-A_2^{(\mathbf s)})^{-1}\|_{L(X_2^{(\mathbf s)})}\le C$. For the symbol, this implies
\[ \big\|\Lambda^{(\mathbf s)}(\xi)\lambda(\lambda-a(\xi))^{-1}\Lambda^{(-\mathbf s)}(\xi)\big\|_{L^\infty(\R^n;\C^{(n+3)\times(n+3)})} \le C.\]
Let $\mathbf s^{(0)} := (2,1,0,0)$. Setting $\xi=(\rho,0,\ldots,0)^\top$ and $\lambda = \lambda_0\rho$ with large $\rho>0$ and fixed $\lambda_0>0$, we obtain
\begin{align}
  \Lambda^{(\mathbf s)}(\xi)& \lambda(\lambda-a(\xi))^{-1}\Lambda^{(-\mathbf s)}(\xi) \nonumber\\
  & = \Lambda^{(\mathbf s -\mathbf s^{(0)})}(\xi) \begin{pmatrix}
    b_0-\lambda_0 I_4 & 0 \\  0 & -\lambda_0 I_{n-1}
  \end{pmatrix}^{-1} \Lambda^{(-\mathbf s +\mathbf s^{(0)})}(\xi)\label{4-2}
\end{align}
modulo lower-order terms with respect to $\rho\to\infty$. Here,
\[ b_0 := \begin{pmatrix}
  0 & 1 & 0 & 0 \\
  -\frac 1\mu & 0 & \frac 1\mu & 0 \\
  0 & -1 & 0 & i\\
  0 & 0 & \frac i\tau & 0
\end{pmatrix}.\]
Direct calculations show that every entry of the matrix $(b^{ij}(\lambda_0))_{i,j=1,\dots,4} := (b_0-\lambda_0I_4)^{-1}$ is a nontrivial rational function of $\lambda_0$ with coefficients depending polynomially on $\frac 1\mu$ and $\frac 1\tau$. Therefore, for every fixed $\mu>0$ and $\tau>0$, we can choose a $\lambda_0>0$ such that every entry $b_{ij}(\lambda_0)$ is non-zero.
For $i,j=1,\dots,4$, the entry of the matrix  \eqref{4-2} at position $(i,j)$ is given by
\[ \langle\xi\rangle ^{s_i-s_i^{(0)}-s_j+s_j^{(0)}} b^{ij}(\lambda_0).\]
Due to $b^{ij}(\lambda_0)\not=0$, we obtain from the boundedness of  \eqref{4-2}
\[ s_i-s_i^{(0)}-s_j+s_j^{(0)} = 0\quad\text{for all }i,j=1,\dots,4.\]
This implies $s_i-s_i^{(0)} = c$ for some $c\in\R$.
\end{proof}

\begin{lemma}
  \label{4.2}
  Let $\tau>0$, $\mu>0$, and $p\in(1,\infty)$. Then $A_p$ does not generate an analytic semigroup in $X_p$.
\end{lemma}

\begin{proof}
  It was shown in \cite{racke-ueda16}, (4.43), that there exists an eigenvalue $\lambda_1(\xi)$ of $a(\xi)$ with $|\Im\lambda_1(\xi)|\to\infty$ and $|\Re\lambda_1(\xi)|\le C$ for $|\xi|\to\infty$. Therefore, the resolvent set of $A_p$ does not contain any sector of the complex plane with angle larger than $\frac\pi 2$. By this, $A_p-\lambda_0$ is not sectorial for any $\lambda_0>0$ which implies that $A_p$ does not generate an analytic semigroup.
\end{proof}

\begin{theorem}
  \label{4.3}
  Let $\tau>0$, $\mu>0$, and $p\in (1,\infty)\setminus\{2\}$. Then the operator $A_p$ generates a $C_0$-semigroup in $X_p$ (and the Cauchy problem \eqref{4-3} is well-posed) if and only if $n=1$.
\end{theorem}

\begin{proof}
  By Corollary~\ref{2.9}, we have to study the symbol $\tilde a(\xi) := \Lambda^{(\mathbf s)}(\xi)a(\xi)\Lambda^{(-\mathbf s)}(\xi)$ with $\mathbf s:=(2,1,0,0)$. We have $\tilde a(\xi) \in S^1_{\cl}(\R^n,\C^{(n+3)\times(n+3)})$ with principal symbol
  \[ \tilde a_0(\xi) = \begin{pmatrix}
    0 & |\xi| & 0 & 0 & 0 & \ldots & 0 \\
    -\frac{|\xi|}\mu & 0 & \frac{|\xi|}\mu & 0 & 0 & \ldots & 0 \\
    0 & -|\xi| & 0 & i\xi_1 & i\xi_2 & \ldots & i\xi_n\\
    0 & 0 & \frac{i\xi_1}\tau & 0 & 0 &\ldots & 0 \\
    \vdots & \vdots & \vdots & \vdots & \vdots & & \vdots  \\
    0 & 0 & \frac{i\xi_n}\tau & 0 & 0 & \ldots & 0
  \end{pmatrix}\]
  (for large $|\xi|$). Its characteristic polynomial is given by
  \[ \det(\lambda-\tilde a_0(\xi)) = \lambda^{n-1}\Big( \lambda^4 + (\tfrac 1\tau + \tfrac 2\mu)|\xi|^2\lambda^2 + \tfrac1{\tau\mu}|\xi|^4\Big).\]
  Therefore, all eigenvalues of $\tilde a_0(\xi)$ are functions of $|\xi|$ and lie on the imaginary axis. If $A_p$ generates a $C_0$-semigroup in $X_p$, then $n=1$ by Remark~\ref{3.8} c).

  Now let $n=1$. We write $\tilde a_0(\xi)$ in the form $\tilde a_0(\xi)= \xi a_+ \chi_{[0,\infty)}(\xi) + \xi a_-\chi_{(-\infty,0)}(\xi)$ with the constant matrices
  \[ a_\pm := \begin{pmatrix}
    0 & \pm 1 & 0 & 0 \\
    \mp \frac 1\mu & 0 & \pm\frac1\mu & 0 \\
    0 & \mp 1& 0 & i\\
    0 & 0 & \frac{i}\tau & 0
  \end{pmatrix}.\]
  We apply Theorem~\ref{2.7} to $e^{t\tilde a_0(\xi)} = e^{t\xi a_+}\chi_{[0,\infty)}(\xi) + e^{t\xi a_-}\chi_{(-\infty,0)}(\xi)$. As both matrices $a_+,a_-$ have four different purely imaginary eigenvalues, they are 	diagonalizable, and Remark~\ref{3.8a} d) yields that $e^{t\xi a_\pm}$ satisfies \eqref{2-5}. On the other hand, the characteristic functions $\chi_{[0,\infty)}$ and $\chi_{(-\infty,0)}$ are $L^p$-Fourier multipliers by Mikhlin's theorem. Therefore, the realization of $\tilde a_0$ generates a $C_0$-semigroup in $X_p$ for $n=1$. As $\tilde a$ is a bounded perturbation of $\tilde a_0$, we see that $A_p$ generates a $C_0$-semigroup in $X_p$ for $n=1$.
\end{proof}

Now we consider the case $\mu=0$. Now the natural setting is $X_p := W_p^2(\R^n)\times L^p(\R^n) \times L^p(\R^n)\times L^p(\R^n;\C^n)$. The maximal domain is given by
\begin{align*}
 D(A_p) & = \big\{ U = (u,v,\theta,q)^\top\in W_p^3(\R^n)\times W_p^2(\R^n)\times W_p^1(\R^n)\times L^p(\R^n;\C^n):\\
 & \quad
\div u\in L^p(\R^n),\, \Delta u+\theta\in W_p^2(\R^n)\big\}.
\end{align*}
In the case $\mu=0$, the results in the $L^2$-case are analog to the case $\mu>0$. However, even for $n=1$, the operator $A_p$ does not generate a $C_0$-semigroup:

\begin{theorem}
  \label{4.4}
  Let $\tau>0$ and $\mu=0$.

  a) Let $p=2$. Then the operator $A_2$ generates a $C_0$-semigroup in $X_2$ but no analytic semigroup.

  b) Let $p\in (1,\infty)\setminus\{2\}$. Then $A_p$ does not generate a $C_0$-semigroup in $X_p$.
\end{theorem}

\begin{proof}
  a) This can be shown in an analog way as for the case $\mu>0$.

  b) Again we have to consider $\tilde a(\xi) = \Lambda^{(\mathbf s)}(\xi) a(\xi) \Lambda^{(-\mathbf s)}(\xi)$ where now $\mathbf s = (2,0,0,0)$. Now we have $\tilde a\in S^2(\R^n;\C^{(n+3)\times (n+3)})$, and the principal symbol is given by
  \[ \tilde a_0(\xi) = \begin{pmatrix}
    0 & |\xi|^2 & 0 & 0 \\
    -|\xi|^2 & 0 & |\xi|^2 & 0 \\
    0 & -|\xi|^2 & 0 & 0\\
    0 & 0 & 0 & 0
  \end{pmatrix}\in \C^{(n+3)\times(n+3)}.\]
  As this matrix has the purely imaginary eigenvalues $\pm \sqrt{2}|\xi|^2i$, Theorem~\ref{3.9} c) shows that $A_p$ is no generator of a $C_0$-semigroup.
\end{proof}

\subsection{Fourier's law}

Now let us consider the case $\tau=0$, i.e., the thermoelastic plate equation with Fourier's law of heat conduction. We first remark that for $\tau=0$ and $\mu=0$, the operator generates an analytic semigroup in the $L^p$-setting, see \cite{denk-racke06}, Theorem~3.5. Therefore, we only have to investigate the case $\mu>0$.

So we consider for $\mu>0$ the equation
\begin{align*}
  u_{tt} + \Delta^2 u - \mu \Delta u_{tt} + \Delta \theta & = 0,\\
  \theta_t -\Delta \theta -\Delta u_t & = 0
\end{align*}
in $\R^n$ with initial conditions $u|_{t=0} = u_0$, $u_t|_{t=0} = u_1$, $\theta|_{t=0}=\theta_0$. Setting $U:= (u,v,\theta)^\top$ with $v:=u_t$, we obtain the Cauchy problem
\begin{equation}
  \label{4-4}
  \big(\partial_t - A(D)\big) U(t) = 0 \; (t>0),\quad U(0)=U_0
\end{equation}
with $U_0 := (u_0,u_1,\theta_0)^\top$ and
\[ A(D) := \begin{pmatrix}
  0 & 1 & 0 \\
  -(1-\mu\Delta)^{-1} \Delta^2 & 0 -(1-\mu\Delta)^{-1}\Delta\\
  0 & \Delta & \Delta
\end{pmatrix}.\]
The natural basic space for the operator related to \eqref{4-4} is given by
\[ X_p := W_p^2(\R^n)\times W_p^1(\R^n)\times L^p(\R^n),\]
and the operator is defined as the realization of $A(D)$ in $X_p$ with maximal domain
\[ D(A_p) = W_p^3(\R^n)\times W_p^2(\R^n)\times W_p^2(\R^n).\]
The symbol of $A(D)$ equals
\[ a(\xi) := \begin{pmatrix}
  0 & 1 & 0 \\
  -\frac{|\xi|^4}{1+\mu|\xi|^2} & 0 & \frac{|\xi|^2}{1+\mu|\xi|^2}\\
  0 & -|\xi|^2 & -|\xi|^2
\end{pmatrix}.\]
Setting $\Lambda(\xi) := \diag( \langle \xi\rangle^2,\langle\xi\rangle, 1)$, we have to study the mixed-order symbol $\tilde a(\xi) := \Lambda(\xi)a(\xi)\Lambda(\xi)^{-1}$. We have $\tilde a \in S_{\cl}^2(\R^n;\C^{3\times 3})$ and, for large $|\xi|$,
\[ \tilde a(\xi) = |\xi|^2 a_0 +|\xi| a_1 + a_2(\xi)\]
with $a_2\in S_{\cl}^0(\R^n;\C^{3\times 3})$ and
\[ a_0 = \begin{pmatrix}
  0 & 0 & 0\\ 0 & 0 & 0 \\ 0 & 0 & -1
\end{pmatrix},\quad
a_1 =\begin{pmatrix}
  0 & 1 & 0 \\ -\frac 1\mu & 0 & \frac 1\mu\\
  0 & -1 & 0
\end{pmatrix}.\]

\begin{theorem}
  \label{4.5}
  Let $\tau=0$ and $\mu>0$.

  a) Let $p=2$. Then the operator $A_2$ generates a $C_0$-semigroup in $X_2$ but no analytic semigroup.

  b) Let $p\in (1,\infty)\setminus\{2\}$. Then $A_p$ generates a $C_0$-semigroup if and only if $n=1$.
\end{theorem}

\begin{proof}
  a) Again, the first statement is a straightforward application of the Lumer-Phillips theorem. For the second statement, we use the fact that $a(\xi)$ has eigenvalues with bounded real part and unbounded imaginary part, see \cite{racke-ueda16}, (4.40).

  b) In contrast to the proof of Theorem~\ref{4.4} b), we cannot apply Theorem~\ref{3.9} c) as the nontrivial eigenvalue of the principal symbol $a_0$ has negative imaginary part. Therefore, we use the idea of an approximate diagonalization procedure which was introduced in \cite{kozhevnikov96}, Section~2.4, on closed manifolds and in \cite{denk-saal-seiler09}, Section~3.3, in $\R^n$.

  By bounded perturbation, the  operator $A_p$ generates a $C_0$-semigroup in $X_p$ if and only if the operator $B_p$ corresponding to the mixed-order symbol
  \[ b(\xi) := |\xi|^2 a_0 + |\xi| a_1 = \left(\begin{array}{ccc} 0 & {|\xi|} & 0\\ -\frac{{|\xi|}}{{\mu}} & 0 & \frac{{|\xi|}}{{\mu}}\\ 0 & - {|\xi|} & - {{|\xi|}}^2 \end{array}\right)\]
  (for large $|\xi|$) generates a $C_0$-semigroup in $L^p(\R^n;\R^3)$. We define the transformation matrix
\[  S := \frac1{\sqrt\mu}\left(\begin{array}{ccc} 0 & \sqrt{{\mu}} & \sqrt{{\mu}}\\ \frac{1}{\sqrt{{\mu}}\, {|\xi|}} & {i} & -{i}\\ - \sqrt{{\mu}} & -\frac{{i}}{{|\xi|}} & \frac{{i}}{{|\xi|}} \end{array}\right)
.\]
As in the proof of Corollary~\ref{2.9}, we see that well-posedness is invariant under similarity transforms $b(\xi)\mapsto S^{-1}(\xi)b(\xi)S(\xi)$. An explicit calculation shows
\[ S^{-1}(\xi) =\frac 1{2(\mu|\xi|^2-1)} \left(\begin{array}{ccc} 0 & - 2\, {\mu}\, {|\xi|} & - 2\, {\mu}\, {{|\xi|}}^2\\ {\mu}\, {{|\xi|}}^2 - 1 & - {{\mu}}^{3/2}\, {{|\xi|}}^2\, {i} & - \sqrt{{\mu}}\, {|\xi|}\, {i}\\ {\mu}\, {{|\xi|}}^2 - 1 & {{\mu}}^{3/2}\, {{|\xi|}}^2\, {i} & \sqrt{{\mu}}\, {|\xi|}\, {i} \end{array}\right)
\]
and
\begin{equation}\label{4-6}
 \tilde b(\xi) := S^{-1}(\xi) b(\xi) S(\xi) = \begin{pmatrix} -|\xi|^2 & 0& 0\\0 & \frac{|\xi|}{\sqrt\mu} i & 0\\ 0 &0&
-\frac{|\xi|}{\sqrt\mu} i\end{pmatrix}  + R(\xi)
\end{equation}
with
\[ R(\xi) =  \frac 1{2(\mu |\xi|^2-1)} \left(\begin{array}{ccc} 2{{|\xi|}}^2 & 2{{|\xi|}}^2 + \frac{{2|\xi|}\, {i}}{\sqrt{{\mu}}} & {{2|\xi|}}^2 - \frac{{2|\xi|}\, {i}}{\sqrt{{\mu}}}\\ |\xi|^2 - \frac{1}{ {\mu}} + \frac{{|\xi|}\, {i}}{\sqrt{{\mu}}} &  -  |\xi|^2  + \frac{{|\xi|}\, {i}}{  \sqrt{{\mu}}} & |\xi|^2 + \frac{{|\xi|}\, {i}}{ \sqrt{{\mu}}}\\ |\xi|^2 - \frac{1}{ {\mu}} - \frac{{|\xi|}\, {i}}{ \sqrt{{\mu}}} & |\xi|^2 - \frac{{|\xi|}\, {i}}{ \sqrt{{\mu}}} &  - |\xi|^2 - \frac{{|\xi|}\, {i}}{  \sqrt{{\mu}}} \end{array}\right)
.\]
We see that $S, S^{-1}\in S^0_{\cl}(\R^n,\C^{3\times 3})$. Therefore, the symbol $S$ induces an isomorphism $S(D)\colon L^p(\R^n;\C^3)\to L^p(\R^n;\C^3)$, and $B_p$ is well-posed in $L^p(\R^n;\C^3)$ if and only if $\tilde B_p := S(D)^{-1} B_p S(D)$ is well-posed in $L^p(\R^n;\C^3)$. Moreover, we have $R\in S^0_{\cl}(\R^n,\C^{3\times 3})$. Therefore, $\tilde B_p$ is a bounded perturbation of the operator related to the diagonal matrix in \eqref{4-6}.

Altogether we have seen that $A_p$ generates a $C_0$-semigroup in $X_p$ if and only if the operator related to the symbol
\[ \begin{pmatrix} -|\xi|^2 & 0& 0\\0 & \frac{|\xi|}{\sqrt\mu} i & 0\\ 0 &0&
-\frac{|\xi|}{\sqrt\mu} i\end{pmatrix}\]
generates a $C_0$-semigroup in $L^p(\R^n,\R^3)$. Now we can consider each component separately. If $A_p$ generates a $C_0$-semigroup, then the eigenvalue $\mu^{-1/2}|\xi|$ has to be a linear function of $\xi$ at the points of differentiability which implies $n=1$. On the other hand, in the case $n=1$ we can write
$|\xi| = -\xi \chi_{(-\infty,0)}(\xi) + \xi \chi_{[0,\infty)}(\xi)$ (cf. the proof of Theorem~\ref{4.3}) and obtain that $A_p$ generates a $C_0$-semigroup.
\end{proof}

\bibliographystyle{abbrv}

\begin{thebibliography}{10}

\bibitem{amann03}
H.~Amann.
\newblock Vector-valued distributions and {F}ourier multipliers.
\newblock Unpublished manuscript (see
  http://user.math.uzh.ch/amann/books.html), 2003.

\bibitem{arendt-batty-hieber-neubrander11}
W.~Arendt, C.~J.~K. Batty, M.~Hieber, and F.~Neubrander.
\newblock {\em Vector-valued {L}aplace transforms and {C}auchy problems},
  volume~96 of {\em Monographs in Mathematics}.
\newblock Birkh\"auser/Springer Basel AG, Basel, second edition, 2011.

\bibitem{bargetz-ortner15}
C.~Bargetz and N.~Ortner.
\newblock Solution to the {C}auchy problem for quasi-hyperbolic systems by
  vector-valued convolution.
\newblock {\em Appl. Anal.}, 94(11):2355--2369, 2015.

\bibitem{brenner66}
P.~Brenner.
\newblock The {C}auchy problem for symmetric hyperbolic systems in {$L_{p}$}.
\newblock {\em Math. Scand.}, 19:27--37, 1966.

\bibitem{brenner73}
P.~Brenner.
\newblock The {C}auchy problem for systems in {$L_{p}$} and {$L_{p,\alpha }$}.
\newblock {\em Ark. Mat.}, 11:75--101, 1973.

\bibitem{coriasco-ruzhansky14}
S.~Coriasco and M.~Ruzhansky.
\newblock Global {$L^p$} continuity of {F}ourier integral operators.
\newblock {\em Trans. Amer. Math. Soc.}, 366(5):2575--2596, 2014.

\bibitem{denk-racke06}
R.~Denk and R.~Racke.
\newblock {$L^p$}-resolvent estimates and time decay for generalized
  thermoelastic plate equations.
\newblock {\em Electron. J. Differential Equations}, pages No. 48, 16 pp.
  (electronic), 2006.

\bibitem{denk-saal-seiler09}
R.~Denk, J.~Saal, and J.~Seiler.
\newblock Bounded {$H_\infty$}-calculus for pseudo-differential
  {D}ouglis-{N}irenberg systems of mild regularity.
\newblock {\em Math. Nachr.}, 282(3):386--407, 2009.

\bibitem{denk-schnaubelt15}
R.~Denk and R.~Schnaubelt.
\newblock A structurally damped plate equation with {D}irichlet-{N}eumann
  boundary conditions.
\newblock {\em J. Differential Equations}, 259(4):1323--1353, 2015.

\bibitem{MR1155843}
S.~G. Gindikin and L.~R. Volevich.
\newblock {\em Distributions and convolution equations}.
\newblock Gordon and Breach Science Publishers, Philadelphia, PA, 1992.
\newblock Translated from the Russian by V. M. Volosov.

\bibitem{grubb95}
G.~Grubb.
\newblock Parameter-elliptic and parabolic pseudodifferential boundary problems
  in global {$L_p$} {S}obolev spaces.
\newblock {\em Math. Z.}, 218(1):43--90, 1995.

\bibitem{hoermander60}
L.~H{\"o}rmander.
\newblock Estimates for translation invariant operators in {$L^{p}$}\ spaces.
\newblock {\em Acta Math.}, 104:93--140, 1960.

\bibitem{hoermander83}
L.~H{\"o}rmander.
\newblock {\em The analysis of linear partial differential operators. {II}},
  volume 257 of {\em Grundlehren der Mathematischen Wissenschaften [Fundamental
  Principles of Mathematical Sciences]}.
\newblock Springer-Verlag, Berlin, 1983.
\newblock Differential operators with constant coefficients.

\bibitem{kozhevnikov96}
A.~Kozhevnikov.
\newblock Asymptotics of the spectrum of {D}ouglis-{N}irenberg elliptic
  operators on a compact manifold.
\newblock {\em Math. Nachr.}, 182:261--293, 1996.

\bibitem{lasiecka-triggiani98}
I.~Lasiecka and R.~Triggiani.
\newblock Analyticity, and lack thereof, of thermo-elastic semigroups.
\newblock In {\em Control and partial differential equations
  ({M}arseille-{L}uminy, 1997)}, volume~4 of {\em ESAIM Proc.}, pages 199--222
  (electronic). Soc. Math. Appl. Indust., Paris, 1998.

\bibitem{lasiecka-triggiani98a}
I.~Lasiecka and R.~Triggiani.
\newblock Two direct proofs on the analyticity of the s.c.\ semigroup arising
  in abstract thermo-elastic equations.
\newblock {\em Adv. Differential Equations}, 3(3):387--416, 1998.

\bibitem{littman63}
W.~Littman.
\newblock The wave operator and {$L_{p}$} norms.
\newblock {\em J. Math. Mech.}, 12:55--68, 1963.

\bibitem{liu-liu97}
K.~Liu and Z.~Liu.
\newblock Exponential stability and analyticity of abstract linear
  thermoelastic systems.
\newblock {\em Z. Angew. Math. Phys.}, 48(6):885--904, 1997.

\bibitem{ortner-wagner90}
N.~Ortner and P.~Wagner.
\newblock Some new fundamental solutions.
\newblock {\em Math. Methods Appl. Sci.}, 12(5):439--461, 1990.

\bibitem{ortner-wagner15}
N.~Ortner and P.~Wagner.
\newblock {\em Fundamental solutions of linear partial differential operators}.
\newblock Springer, Cham, 2015.
\newblock Theory and practice.

\bibitem{peral80}
J.~C. Peral.
\newblock {$L^{p}$} estimates for the wave equation.
\newblock {\em J. Funct. Anal.}, 36(1):114--145, 1980.

\bibitem{quintanilla-racke08}
R.~Quintanilla and R.~Racke.
\newblock Qualitative aspects of solutions in resonators.
\newblock {\em Arch. Mech. (Arch. Mech. Stos.)}, 60(4):345--360, 2008.

\bibitem{racke-ueda16}
R.~Racke and Y.~Ueda.
\newblock Dissipative structures for thermoelastic plate equations in {$\Bbb
  R^n$}.
\newblock {\em Adv. Differential Equations}, 21(7-8):601--630, 2016.

\bibitem{rahman-schmeisser02}
Q.~I. Rahman and G.~Schmeisser.
\newblock {\em Analytic theory of polynomials}, volume~26 of {\em London
  Mathematical Society Monographs. New Series}.
\newblock The Clarendon Press, Oxford University Press, Oxford, 2002.

\bibitem{ruzhansky01}
M.~Ruzhansky.
\newblock {\em Regularity theory of {F}ourier integral operators with complex
  phases and singularities of affine fibrations}, volume 131 of {\em CWI
  Tract}.
\newblock Stichting Mathematisch Centrum, Centrum voor Wiskunde en Informatica,
  Amsterdam, 2001.

\bibitem{said-houari13}
B.~Said-Houari.
\newblock Decay properties of linear thermoelastic plates: {C}attaneo versus
  {F}ourier law.
\newblock {\em Appl. Anal.}, 92(2):424--440, 2013.

\bibitem{ueda-duan-kawashima12}
Y.~Ueda, R.~Duan, and S.~Kawashima.
\newblock Decay structure for symmetric hyperbolic systems with non-symmetric
  relaxation and its application.
\newblock {\em Arch. Ration. Mech. Anal.}, 205(1):239--266, 2012.

\bibitem{wong99}
M.~W. Wong.
\newblock {\em An introduction to pseudo-differential operators}.
\newblock World Scientific Publishing Co., Inc., River Edge, NJ, second
  edition, 1999.

\end{thebibliography}

\end{document}